\documentclass[12pt]{amsart}
\usepackage{amsmath,amsthm,amssymb}
\usepackage{graphicx}

\begin{document}
\newtheorem{thm}{Theorem}[section]
\newtheorem{cor}{Corollary}[section]
\newtheorem{lem}{Lemma}[section]
\newtheorem{prop}{Proposition}[section]
\theoremstyle{definition}
\newtheorem{defi}{Definition}[section]
\newtheorem{exa}{Example}[section]
\newtheorem{rem}{Remark}[section]

\numberwithin{equation}{section}

\begin{center}
 Box-counting dimension of solution curves for a class of
 two-dimensional nonautonomous linear differential systems
\end{center}

\vspace{2ex}

\begin{center}
 Masakazu Onitsuka \\[1ex]
 Department of Applied Mathematics, Faculty of Science \\
 Okayama University of Science \\
 Ridaichou 1--1, Okayama 700--0005, Japan \\
 email: \texttt{onitsuka@xmath.ous.ac.jp} \\[1ex]
 and \\[1ex]
 Satoshi Tanaka 
 \footnote{This work was supported by JSPS KAKENHI Grant Number 26400182.
 \\ \hfill \today}
 \\[1ex]
 Department of Applied Mathematics, Faculty of Science \\
 Okayama University of Science \\
 Ridaichou 1--1, Okayama 700--0005, Japan \\
 email: \texttt{tanaka@xmath.ous.ac.jp}
\end{center}

\vspace{3ex}

{\bf Abstract.}
The two-dimensional linear differential system
\begin{equation*}
 x' = y, \quad y' = -x-h(t)y
\end{equation*}
is considered on $[t_0,\infty)$, where $h \in C^1[t_0,\infty)$ and
$h(t)>0$ for $t \ge t_0$.
The box-counting dimension of the graphs of solution curves is calculated.
Criteria to obtain the box-counting dimension of spirals are also
established.

\vspace{2ex}

\noindent{2010} {\itshape Mathematical Subject Classification.}
34A30, 37C45, 28A80 \\
\noindent{\itshape Keywords.}
linear system, box-counting dimension, spiral

\vspace{1ex}

% Section 1

\section{Introduction}

In this paper, we consider the following two-dimensional linear
differential system
\begin{equation}
   \begin{array}{l}
    x' = y, \\[1ex]
    y' = -x-h(t)y
   \end{array}
 \label{S}
\end{equation}
for $t \ge t_0$, where $h \in C^1[t_0,\infty)$ and $h(t)>0$ for $t \ge t_0$.
This system has the {\it zero solution} $(x(t),y(t))\equiv(0,0)$.
Setting $y=x'$, we can rewrite \eqref{S} as the damped linear oscillator
\begin{equation}
 x'' + h(t) x' + x = 0, \quad t \ge t_0.
  \label{E}
\end{equation}
By a general theory (for example \cite{Cop, Har}),
there exists a unique solution of \eqref{S} on $[t_0,\infty)$ with the
initial condition $x(t_1)=\alpha$ and $y(t_1)=\beta$ for every $\alpha$, 
$\beta \in {\bf R}$ and $t_1 \ge t_0$.
Hence, we note that every nontrivial solution $(x(t),y(t))$
satisfies $(x(t),y(t))\ne(0,0)$ for $t\ge t_0$.

The zero solution $(x(t),y(t))\equiv(0,0)$ of \eqref{S} is said to be 
{\it attractive} if every solution $(x(t),y(t))$ of \eqref{S} satisfies 
$\lim_{t\to\infty} x(t) = \lim_{t\to\infty} y(t) =0$. 
There are a lot of studies of the attractivity to \eqref{S}
(see, for example, \cite{DIST,Onitsuka2010,Onitsuka2011,Smith,SO10}).

Now, we assume that the zero solution of \eqref{S} is attractive.
Let $(x(t),y(t))$ be a solution of \eqref{S}.
We define the solution curve of $(x(t),y(t))$ on $[t_1,\infty)$ in
${\bf R}^2$ by
\[
 \Gamma_{(x,y;t_1)} = \{ (x(t),y(t)) : t \ge t_1 \}
\]
for each fixed $t_1 \ge t_0$.
A curve $\Gamma_{(x,y;t_1)}$ is said to be {\it simple} if 
$(x(t),y(t)) \ne (x(s),y(s))$ for $t$, $s \in [t_1,\infty)$
with $t \ne s$.
A simple solution curve $\Gamma_{(x,y;t_1)}$ is said to be 
{\it rectifiable} if the length of $\Gamma_{(x,y;t_1)}$ is finite, that is
\[
 \int_{t_1}^\infty \sqrt{|x'(t)|^2+|y'(t)|^2} dt < \infty.
\]
Otherwise, it is said to be {\it non-rectifiable}, that is
\[
 \int_{t_1}^\infty \sqrt{|x'(t)|^2+|y'(t)|^2} dt = \infty.
\]

The rectifiability of solutions to two-dimensional linear differential
systems was studied by Mili\v{c}i\'{c} and Pa\v{s}i\'{c} \cite{MP}
and Naito and Pa\v{s}i\'{c} \cite{NP}.
Naito, Pa\v{s}i\'{c} and Tanaka \cite{NPT} obtained rectifiable and
non-rectifiable results of solutions to half-linear differential systems.
Recently, the following Theorem A is established in \cite{OT2017}.
In what follows, the following notation will be used:
\[
 H(t) = \int_{t_0}^t h(s) ds.
\]

\bigskip

\noindent{\bf Theorem A.} \it
 Let $h \in C^1[t_0,\infty)$ satisfy $h(t)>0$ for $t \ge t_0$.
 Assume that the following conditions \eqref{H(oo)=oo} and \eqref{HW}
 are satisfied\textup{:}
 \begin{gather}
  \int_{t_0}^\infty h(t) dt = \infty; \label{H(oo)=oo} \\
  \int_{t_0}^\infty | 2 h'(t) + |h(t)|^2 | dt < \infty.
  \label{HW}
 \end{gather}
 Then, the zero solution of \eqref{S} is attractive and every nontrivial 
 solution $(x(t),y(t))$ of \eqref{S} is a spiral, rotating in a
 clockwise direction for all sufficiently large $t \ge t_0$, and its
 solution curve $\Gamma_{(x,y;t_0)}$ is simple.
 Moreover, the following properties \textup{(i)} and \textup{(ii)}
 hold\textup{:}
 \begin{enumerate}
  \item every nontrivial solution of \eqref{S} is rectifiable if 
	\begin{equation*}
	 \int_{t_0}^\infty e^{-H(t)/2} dt < \infty;
	\end{equation*}
  \item every nontrivial solution of \eqref{S} is non-rectifiable if
	\begin{equation*}
	 \int_{t_0}^\infty e^{-H(t)/2} dt = \infty.
	\end{equation*}
 \end{enumerate}
\rm

\bigskip

In the above theorem, we adopt the definition of a spiral,
according to a celebrated book by Hartman
\cite[Chapters VII and VIII]{Har} as follows.
For every nontrivial solution $(x(t),y(t))$ of \eqref{S}, we introduce
polar coordinates
\[
 x(t) = r(t) \cos \theta(t), \quad
 y(t) = r(t) \sin \theta(t),
\]
where the amplitude $r(t)>0$.
A nontrivial solution $(x(t),y(t))$ of \eqref{S} is said to be a 
{\it spiral} if $|\theta(t)| \to \infty$ as $t \to \infty$.  

In this paper, we obtain the box-counting dimension of the solution curve
$\Gamma_{(x,y;t_1)}$ for a nontrivial solution $(x(t),y(t))$ of \eqref{S}.
For a bounded subset $\Gamma$ of ${\bf R}^2$, we define the
{\it box-counting dimension} ({\it Minkowski-Bouligand dimension}) of
$\Gamma$ by
\begin{equation*}
 \dim_{\rm B} \Gamma 
  = 2 - \lim_{\varepsilon \to +0} \frac{\log|\Gamma_\varepsilon|}
 {\log \varepsilon}, 
\end{equation*}
where $\Gamma_\varepsilon$ denotes the $\varepsilon$-neighborhood of 
$\Gamma$ defined by
\begin{equation}
 \Gamma_\varepsilon 
  = \{ (x,y) \in {\bf R}^2 : d((x,y),\Gamma) \le \varepsilon \},
  \label{Ge}
\end{equation}
$d((x,y),\Gamma)$ denotes the Euclidean distance from $(x,y)$
to $\Gamma$, and $|\Gamma_{\varepsilon}|$ denotes the two-dimensional
Lebesgue measure of $\Gamma_{\varepsilon}$. 
More details on the definition of the box-counting dimension can be
found in Falconer \cite{Fal} and Tricot \cite{Tri}.
If there exist $d \in [0,2]$, $c_1>0$ and $c_2>0$ such that
\[
 c_1 \varepsilon^{2-d} \le |\Gamma_\varepsilon| \le c_2 \varepsilon^{2-d}
\]
for each sufficiently small $\varepsilon>0$, then 
$\dim_{\rm B} \Gamma=d$.

The following result has been established in
Tricot \cite[\S9.1, Theorem]{Tri}.

\begin{prop}\label{length}
 Let $\Gamma$ be a simple curve of finite length.
 Then,
 \[
  \lim_{\varepsilon\to+0} \frac{|\Gamma_\varepsilon|}{2\varepsilon}
  = \mbox{\rm length}(\Gamma),
 \]
 where $\mbox{\rm length}(\Gamma)$ denotes the length of $\Gamma$.
\end{prop}

Therefore, if $\mbox{\rm length}(\Gamma)<\infty$, then
$\dim_{\rm B} \Gamma=1$.

The box-counting dimensions of the graph of solutions of the nonautonomous
differential equation was first obtained by Pa\v{s}i\'{c} \cite{Pasic2003}.
Thereafter, it is obtained about the nonautonomous second order linear
differential equations in \cite{KPW08, Pasic08, PT11, PT13}.
On the other hands, the box-counting dimensions of solution curves to
autonomous two-dimensional nonlinear differential systems are
established in \cite{PZZ2009, RZZ2012, ZZ2005, ZZ2008}.
Recently, Korkut, Vlah and \v{Z}upanovi\'{c} \cite{KVZ} consider the
equation
\begin{equation}
 t^2x'' + t(2-\mu)x'+(t^2-\nu^2)x=0,
  \label{Bessel}
\end{equation}
where $\mu$, $\nu \in {\bf R}$, and define generalized Bessel functions 
$\widetilde{J}_{\nu,\mu}$ and $\widetilde{Y}_{\nu,\mu}$ by two linearly
independent solutions of \eqref{Bessel}.
When $\mu=1$, equation \eqref{Bessel} is known as Bessel's differential
equation and Bessel functions $J_\nu$ and $Y_\nu$ are its two linearly
independent solutions.
In \cite{KVZ}, the relation 
\[
 \widetilde{J}_{\nu,\mu}(t)=t^{\frac{\mu-1}{2}} J_{\widetilde{\nu}}(t), \quad
 \widetilde{Y}_{\nu,\mu}(t)=t^{\frac{\mu-1}{2}} Y_{\widetilde{\nu}}(t), \quad
 \widetilde{\nu} = \sqrt{\left(\frac{\mu-1}{2}\right)^2+\nu^2} .
\]
is found, and the following result is established.

\bigskip

\noindent{\bf Theorem B ([6]).} \it
Let $\mu \in (0,2)$, $\nu \in {\bf R}$ and $t_0>0$.
Let $x(t)=\widetilde{J}_{\nu,\mu}(t)$ or $\widetilde{Y}_{\nu,\mu}(t)$.
Then the planar curve $\Gamma=\{(x(t),x'(t)) : t \ge t_0 \}$ satisfies
$\dim_{\rm B} \Gamma = 4/(4-\mu)$.
\rm

\bigskip

It is worth while to note that if $x(t)=\widetilde{J}_{\nu,\mu}(t)$ or
$\widetilde{Y}_{\nu,\mu}(t)$, then $(x(t),y(t)):=(x(t),x'(t))$ is a
solution of the linear differential system
\begin{equation} 
 \begin{array}{l}
  x' = y, \\[1ex]
  y' = -\displaystyle\left( 1 - \frac{\nu^2}{t^2} \right)x-\frac{2-\mu}{t}y.
 \end{array}
 \label{Bessel_system} 
\end{equation}

The following two results are the main results of this paper.

\begin{thm}\label{F}
 Let $h \in C^1[t_0,\infty)$ satisfy $h(t)>0$ for $t \ge t_0$.
 Assume that \eqref{HW} and the following conditions are satisfied\textup{:}
 \begin{gather}
  \limsup_{t\to\infty} t h(t) < \infty;
   \label{th(t)<} \\
  H(t) = 2\alpha\log t + O(1) \quad \textup{\it as} \ t \to \infty \quad
  \textup{\it for\ some}\ \alpha \in (0,1).
   \label{H(t)=2alog}
 \end{gather}
 Then, for every nontrivial solution $(x(t),y(t))$ of \eqref{S}, there
 exists $t_1 \ge t_0$ such that $\dim_{\rm B}\Gamma_{(x,y;t_1)}=2/(1+\alpha)$.
\end{thm}

Here and hereafter, $f(t)=O(1)$ as $t \to \infty$ means that 
there exist $M>0$ and $t_1$ such that $|f(t)| \le M$ for $t \ge t_1$.

\begin{thm}\label{F1}
 Let $h \in C^1[t_0,\infty)$ satisfy $h(t)>0$ for $t \ge t_0$.
 Assume that \eqref{HW} and the following condition are satisfied\textup{:}
 \begin{equation}
  H(t) = 2\log t + O(1) \quad \textup{\it as} \ t \to \infty.
   \label{H(t)=2log}
 \end{equation}
 Then, for every nontrivial solution $(x(t),y(t))$ of \eqref{S}, there exists
 $t_1 \ge t_0$ such that $\dim_{\rm B} \Gamma_{(x,y;t_1)}=1$.
\end{thm}

\begin{exa}
 We consider the case where $h(t)=\lambda t^{-\gamma}$, $\lambda>0$, 
 $1/2<\gamma\le1$ and $t_0 =1$.
 It is easy to check that \eqref{H(oo)=oo} and \eqref{HW} are satisfied, and
 \[
  H(t) = \left\{ 
   \begin{array}{ll}
    \displaystyle\frac{\lambda}{1-\gamma} (t^{1-\gamma} -1), 
    & \displaystyle\frac{1}{2}<\gamma<1, \\[2ex]
    \lambda \log t, & \gamma = 1.
   \end{array}
  \right.
 \] 
 Theorem A implies that the zero solution of \eqref{S} is
 attractive and every nontrivial solution $(x(t),y(t))$ of \eqref{S} is
 a spiral, rotating in a clockwise direction on $[t_1,\infty)$ for some
 $t_1 \ge t_0$, and its solution curve $\Gamma_{(x,y;t_0)}$ is simple 
 and that every nontrivial solution of \eqref{S} is rectifiable when
 either $1/2<\gamma<1$ or $\gamma=1$ and $\lambda>2$, and every
 nontrivial solution of \eqref{S} is non-rectifiable when $\gamma=1$ and 
 $0<\lambda \le 2$.
 Let $(x(t),y(t))$ be a nontrivial solution of \eqref{S}.
 Therefore, by Proposition \ref{length}, if either $1/2<\gamma<1$ or
 $\gamma=1$ and $\lambda>2$, then $\dim_{\rm B}\Gamma_{(x,y;t_1)}=1$.
 Moreover, Theorem \ref{F1} implies that
 $\dim_{\rm B}\Gamma_{(x,y;t_2)}=1$ for some $t_2 \ge t_1$ when
 $\gamma=1$ and $\lambda=2$.
 Applying Theorem \ref{F}, we conclude that if $\gamma=1$ and $0<\lambda<2$,
 then there exists $t_2 \ge t_1$ such that
 $\dim_{\rm B} \Gamma_{(x,y;t_2)}=4/(2+\lambda)$.

 Now, we set either 
 $(x(t),y(t))=(\widetilde{J}_{0,2-\lambda}(t),\widetilde{J}_{0,2-\lambda}'(t))$
 or
 $(x(t),y(t))=$\linebreak
 $(\widetilde{Y}_{0,2-\lambda}(t),\widetilde{Y}_{0,2-\lambda}'(t))$,
 where $0<\lambda<2$.
 Recalling that $(\widetilde{J}_{\nu,\mu}(t),\widetilde{J}_{\nu,\mu}'(t))$
 and $(\widetilde{Y}_{\nu,\mu}(t),\widetilde{Y}_{\nu,\mu}'(t))$ are
 solutions of system \eqref{Bessel_system}, we find that $(x(t),y(t))$
 is a solution of \eqref{S} with $h(t)=\lambda t^{-1}$.

 Here, we give numerical simulations of solution curves.

 \medskip

\begin{center}
 Solution curves for the case where $h(t)=\lambda t^{-\gamma}$: \\[2ex]
 \begin{tabular}{cc}
  \includegraphics[width=6cm]{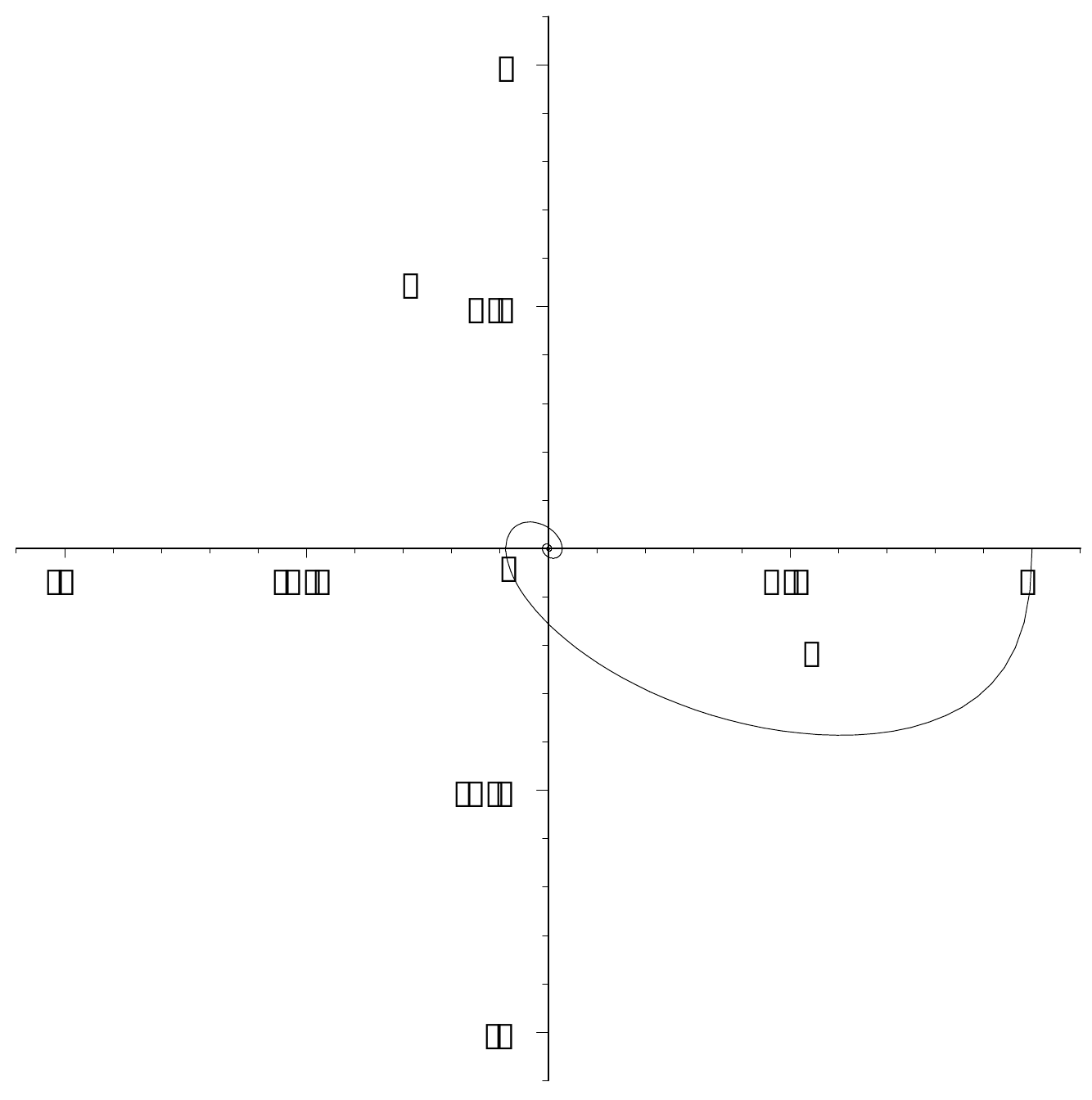} &
  \includegraphics[width=6cm]{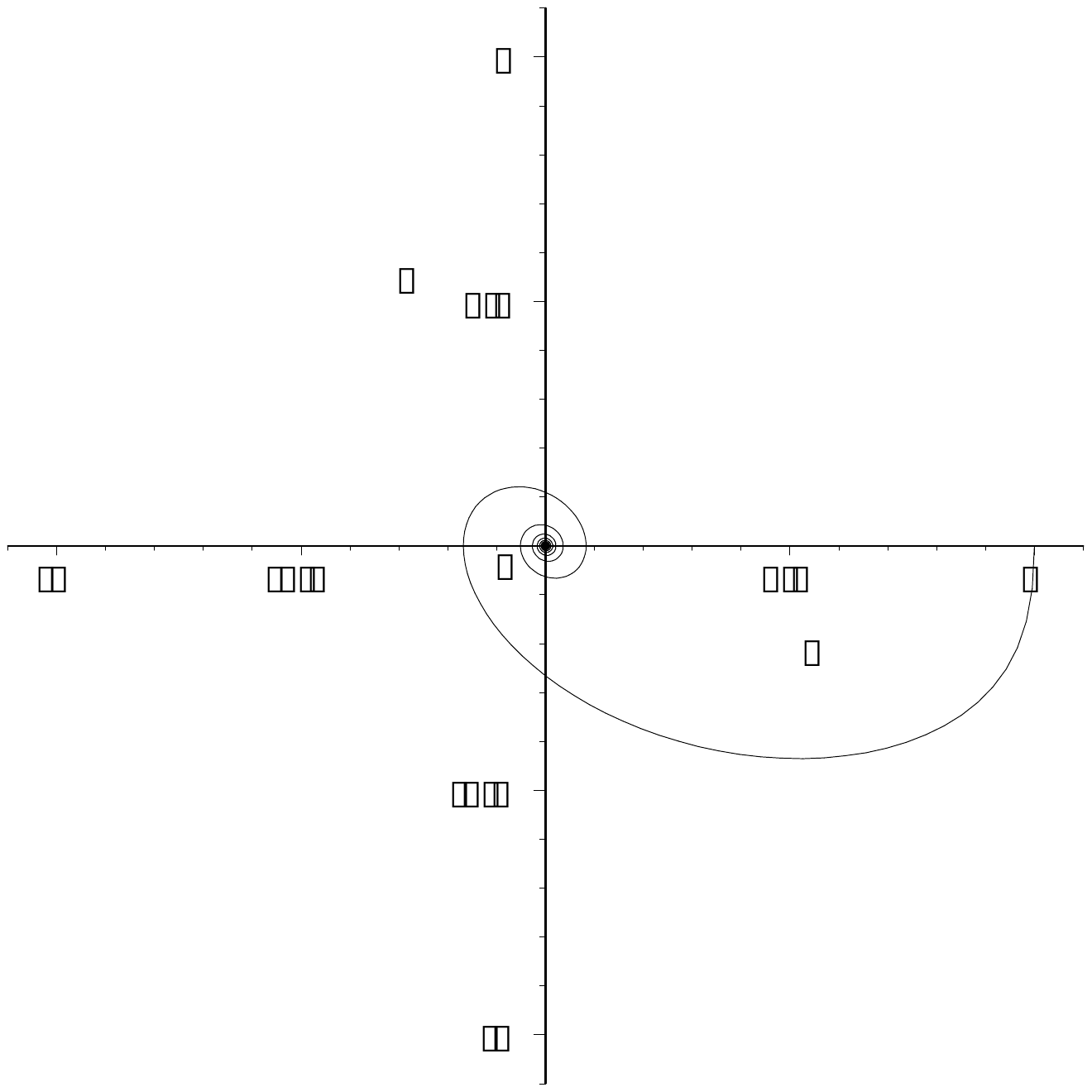} \\[0.5ex]
  \scriptsize{$h(t)=3t^{-3/4}$} & \scriptsize{$h(t)=3t^{-1}$} \\
  \scriptsize{$\dim_{\rm B}\Gamma_{(x,y;t_1)}=1$, rectifiable} &
  \scriptsize{$\dim_{\rm B}\Gamma_{(x,y;t_1)}=1$, rectifiable} 
 \end{tabular}
\end{center}
\begin{center}
 \begin{tabular}{cc}
  \includegraphics[width=6cm]{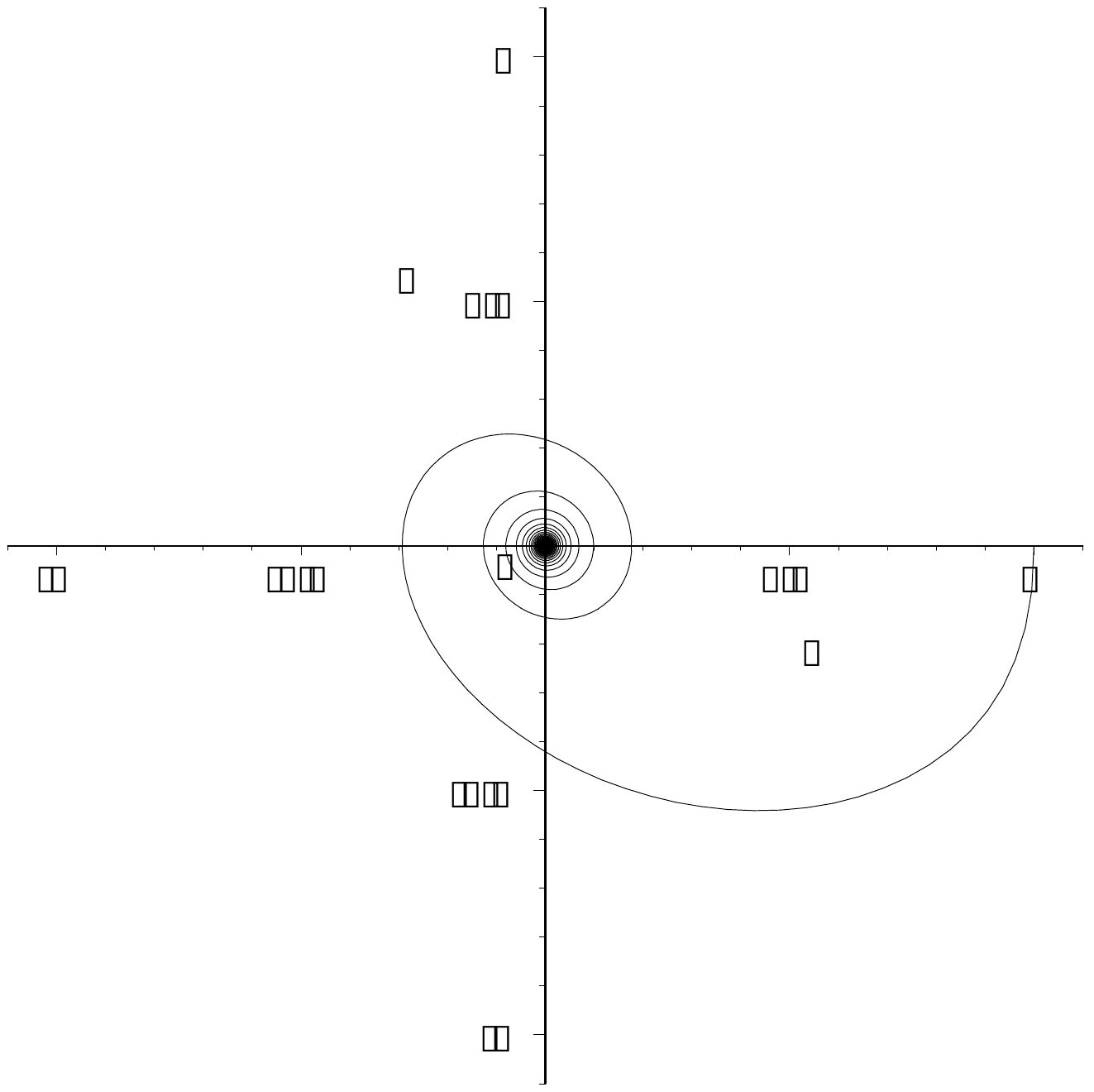} &
  \includegraphics[width=6cm]{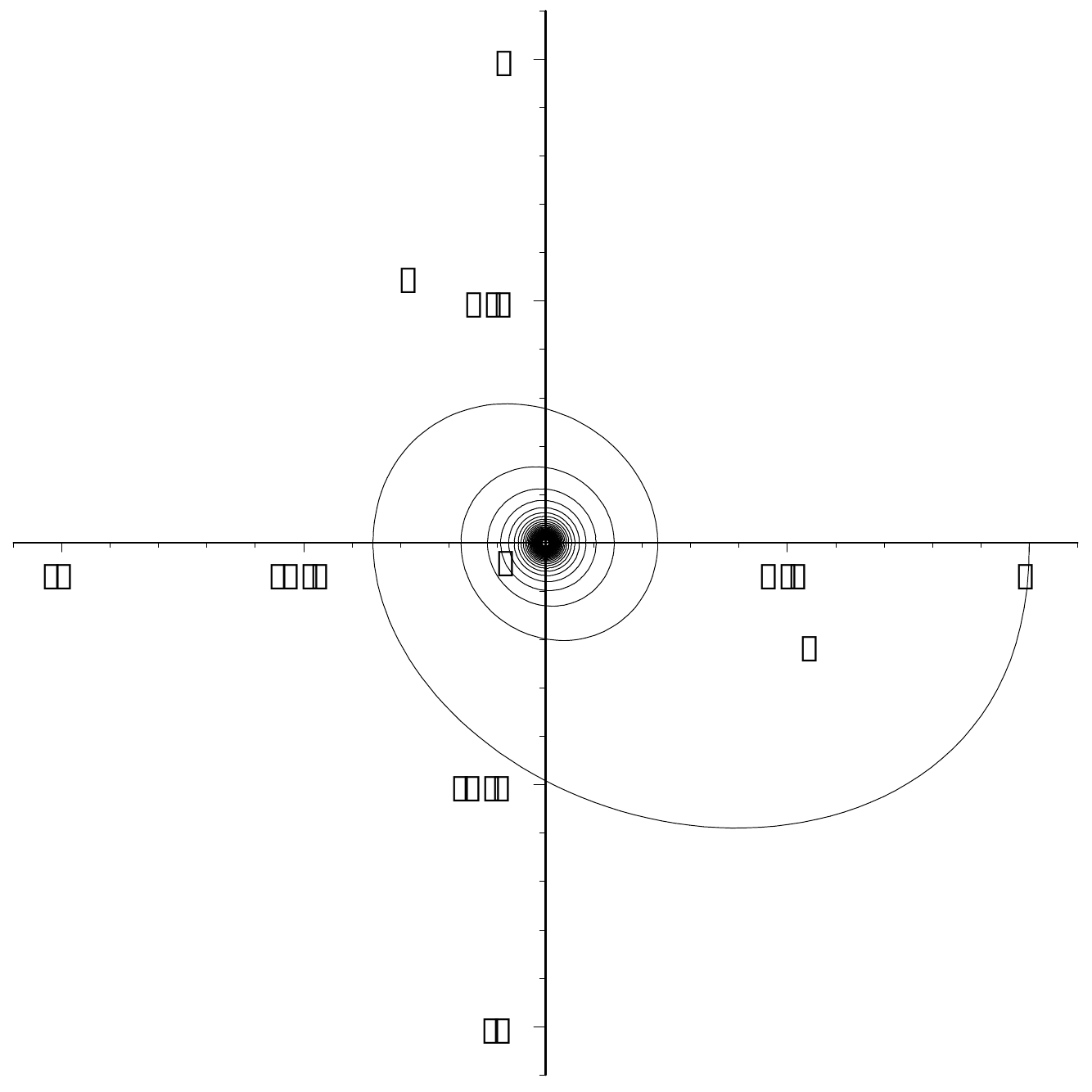} \\[0.5ex]
  \scriptsize{$h(t)=2t^{-1}$} & 
  \scriptsize{$h(t)=(5/3)t^{-1}$} \\
  \scriptsize{$\dim_{\rm B}\Gamma_{(x,y;t_2)}=1$, non-rectifiable} & 
  \scriptsize{$\dim_{\rm B} \Gamma_{(x,y;t_2)}=12/11$, non-rectifiable}
 \end{tabular}
\end{center}

\bigskip

\begin{center}
 \begin{tabular}{cc}
  \includegraphics[width=6cm]{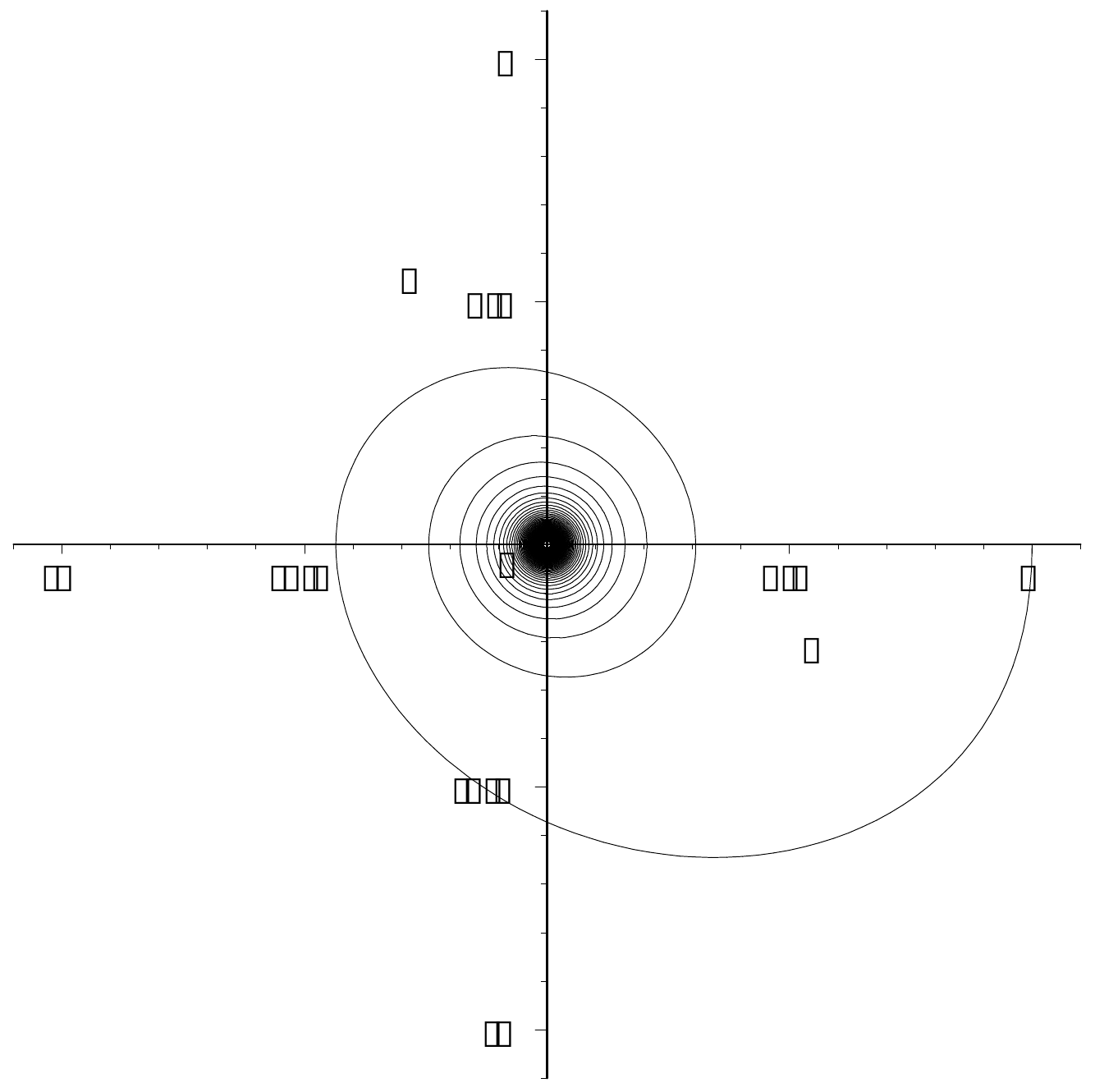} &
  \includegraphics[width=6cm]{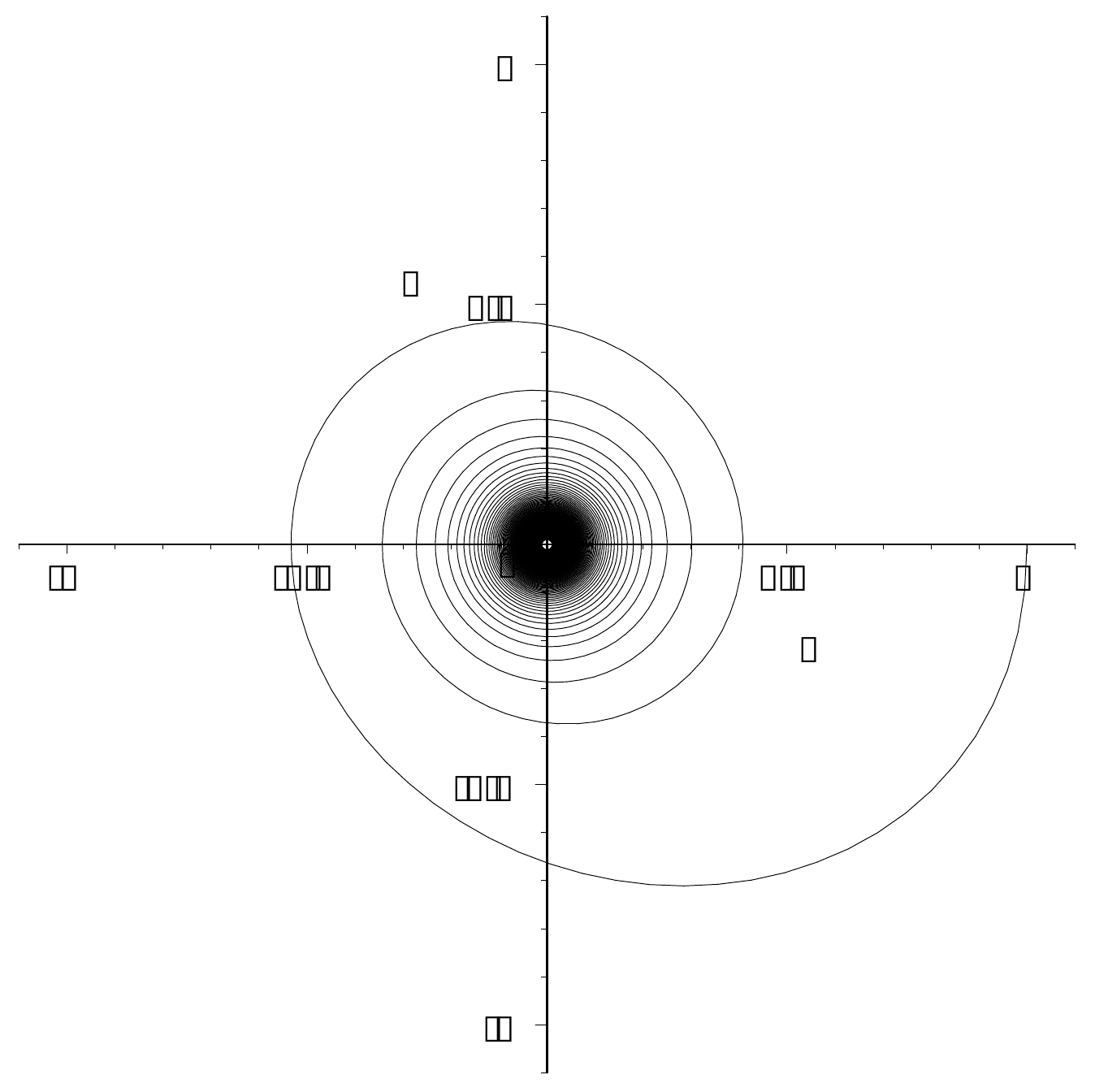} \\[0.5ex]
  \scriptsize{$h(t)=(4/3)t^{-1}$} & \scriptsize{$h(t)=t^{-1}$} \\
  \scriptsize{$\dim_{\rm B}\Gamma_{(x,y;t_2)}=6/5$, non-rectifiable} & 
  \scriptsize{$\dim_{\rm B} \Gamma_{(x,y;t_2)}=4/3$, non-rectifiable}
 \end{tabular}
\end{center}

\end{exa}

\bigskip

The box-counting dimension of the graph of the spiral
$r=\varphi^{-\alpha}$, $\varphi \ge \varphi_1>0$ in polar coordinates
is $2/(1+\alpha)$ when $0<\alpha<1$
(see, for example, Tricot \cite[\S10.4]{Tri}).
\v{Z}ubrini\'{c} and \v{Z}upanovi\'{c} \cite[Theorem 5]{ZZ2005}
generalized this fact to the function $r=f(\varphi)$, $\varphi \ge \varphi_1$.
Korkut, Vlah, \v{Z}ubrini\'{c} and \v{Z}upanovi\'{c}
\cite[Therem 2]{KVZZ} improved this result.
See also Korkut, Vlah and \v{Z}upanovi\'{c} \cite[Theorem 2]{KVZ}.
In this paper, we give the following alternative criterion of the
dimension of spirals.

\pagebreak

\begin{thm}\label{criterion}
 Let $\varphi_1>0$ and let $f \in C[\varphi_1,\infty)$ satisfy
 $\lim_{\varphi\to\infty} f(\varphi)=0$.
 Assume that there exist positive constants
 $\underline{m}$, $\overline{a}$, $M$ and $\alpha \in(0,1)$ such that,
 for all $\varphi \ge \varphi_1$,
 \begin{gather*}
  \underline{m} \varphi^{-\alpha} \le f(\varphi), \\
  0 < f(\varphi) - f(\varphi+2\pi) \le \overline{a} \varphi^{-\alpha-1}, \\
  \mbox{\rm length}(\Gamma(\varphi_1,\varphi))
  \le M \varphi^{1-\alpha}.
 \end{gather*}
 Let $\Gamma$ be the graph of $r=f(\varphi)$ in polar coordinates, that is
 \[
  \Gamma = \{ (f(\varphi)\cos \varphi,f(\varphi)\sin \varphi) :
  \varphi \ge \varphi_1 \}.
 \] 
 Then, $\dim_{\rm B} \Gamma = 2/(1+\alpha)$.
\end{thm}

From Theorem \ref{criterion}, we have the following Corollary.

\begin{cor}\label{cor_criterion}
 Let $\varphi_1>0$ and let $f \in C^1[\varphi_1,\infty)$ satisfy
 $\lim_{\varphi\to\infty} f(\varphi)=0$.
 Assume that there exist positive constants
 $\underline{m}$, $K$ and $\alpha \in(0,1)$ such that,
 for all $\varphi \ge \varphi_1$,
 \begin{gather*}
  \underline{m} \varphi^{-\alpha} \le f(\varphi), \\
  - K \varphi^{-\alpha-1} \le f'(\varphi) \le 0.
 \end{gather*}
 Assume, moreover, that $f'(\varphi)\not\equiv 0$ on $[\varphi,\varphi+2\pi)$
 for each fixed $\varphi \ge \varphi_1$.
 Let $\Gamma= \{ (f(\varphi)\cos \varphi,f(\varphi)\sin \varphi) :
  \varphi \ge \varphi_1 \}$.
 Then, $\dim_{\rm B} \Gamma = 2/(1+\alpha)$.
\end{cor}

The proof of Corollary \ref{cor_criterion} will be given in Section 2.
Using Corollary \ref{cor_criterion}, we prove Theorem \ref{F} in Section 4.
Corollary \ref{cor_criterion} is similar to the criterion by Korkut,
Vlah, \v{Z}ubrini\'{c} and \v{Z}upanovi\'{c} \cite[Therem 2]{KVZZ}.
The proof of Theorem 2 in \cite{KVZZ} is based on the proof of Theorem 5
in \cite{ZZ2005}.
\v{Z}ubrini\'{c} and \v{Z}upanovi\'{c} employed the radial box dimension
to prove Theorem 5 in \cite{ZZ2005}.
On the other hand, the proof of Theorem \ref{criterion}, which will be given
in Section 2, is more direct.

The box-counting dimension of the graph of the spiral
$r=\varphi^{-1}$, $\varphi \ge \varphi_1>0$ in polar coordinates is $1$
(see Tricot \cite[\S10.4]{Tri}).
We generalize this fact as follows.

\begin{thm}\label{criterion_one}
 Let $\varphi_1>1$ and let $f \in C[\varphi_1,\infty)$ satisfy
 $\lim_{\varphi\to\infty} f(\varphi)=0$.
 Assume that there exist positive constants
 $\overline{m}$ and $M$ such that, for all $\varphi \ge \varphi_1$,
 \begin{gather*}
  0 < f(\varphi) \le \overline{m} \varphi^{-1}, \\
  0 < f(\varphi) - f(\varphi+2\pi), \\
  \mbox{\rm length}(\Gamma(\varphi_1,\varphi)) \le M \log \varphi.
 \end{gather*}
 Let $\Gamma= \{ (f(\varphi)\cos \varphi,f(\varphi)\sin \varphi) :
  \varphi \ge \varphi_1 \}$.
 Then, $\dim_{\rm B} \Gamma = 1$.
\end{thm}

From Theorem \ref{criterion_one}, the following corollary follows.

\begin{cor}\label{criterion_one_cor}
 Let $\varphi_1>1$ and let $f \in C[\varphi_1,\infty)$ satisfy
 $\lim_{\varphi\to\infty} f(\varphi)=0$.
 Assume that there exist positive constants
 $\overline{m}$ and $K$ such that, for all $\varphi \ge \varphi_1$,
 \begin{gather*}
  0 < f(\varphi) \le \overline{m} \varphi^{-1}, \\
  - K \varphi^{-1} \le f'(\varphi) \le 0.
 \end{gather*}
 Assume, moreover, that $f'(\varphi)\not\equiv 0$ on $[\varphi,\varphi+2\pi)$
 for each fixed $\varphi \ge \varphi_1$.
 Let $\Gamma= \{ (f(\varphi)\cos \varphi,f(\varphi)\sin \varphi) :
  \varphi \ge \varphi_1 \}$.
 Then, $\dim_{\rm B} \Gamma = 1$.
\end{cor}

The proofs of Theorem \ref{criterion_one} and Corollary
\ref{criterion_one_cor} will be given in Section 3.

% Section 2

\section{Box-counting dimension of spirals}

In this section we prove Theorem \ref{criterion} and Corollary
\ref{cor_criterion}.
First, we give a lemma.

\begin{lem}\label{f<}
 Let $\varphi_1>0$ and let $f \in C[\varphi_1,\infty)$ satisfy
 $f(\varphi)>0$ for $\varphi\ge\varphi_1$ and
 $\lim_{\varphi\to\infty} f(\varphi)=0$.
 Assume that there exist positive constants $\overline{a}$ and $\alpha
 \in(0,1)$ such that
 \begin{gather*}
  0 < f(\varphi) - f(\varphi+2\pi) \le \overline{a} \varphi^{-\alpha-1},
  \quad \varphi \ge \varphi_1.
 \end{gather*}
 Then, there exists a positive constant $\overline{m}$ such that
 $f(\varphi) \le \overline{m} \varphi^{-\alpha}$ for 
 $\varphi \ge \varphi_1$.
\end{lem}

\begin{proof}
 Let $\varphi \ge \varphi_1$.
 Then, there exist $N \in {\bf N}\cup\{0\}$ and
 $\varphi_0 \in [\varphi_1,\varphi_1+2\pi)$ such that
 $\varphi=\varphi_0+2N\pi$.
 Let $n \in {\bf N}$ with $n>N$.
 It follows that
 \begin{align*}
  f(\varphi)
   & = f(\varphi_0+2N\pi) \\
   & = f(\varphi_0+2(n+1)\pi)
     + \sum_{k=N}^n [ f(\varphi_0+2k\pi) - f(\varphi_0+2(k+1)\pi) ] \\
   & \le f(\varphi_0+2(n+1)\pi)
     + \sum_{k=N}^n \overline{a} (\varphi_0+2k\pi)^{-\alpha-1}.
 \end{align*}
 Since
 \begin{align*}
 \frac{(\varphi_0+2k\pi)^{-\alpha-1}}{(\varphi_0+2(k+1)\pi)^{-\alpha-1}}
   & = \left( \frac{\varphi_0+2(k+1)\pi}{\varphi_0+2k\pi} \right)^{\alpha+1} \\
   & = \left( 1 + \frac{2\pi}{\varphi_0+2k\pi} \right)^{\alpha+1} \\
   & \le \left( 1 + \frac{2\pi}{\varphi_1} \right)^{\alpha+1},
   \quad k \in {\bf N} \cup \{0\},
 \end{align*}
 we have
 \[
  (\varphi_0+2k\pi)^{-\alpha-1} \le M_1 (\varphi_0+2(k+1)\pi)^{-\alpha-1},
  \quad k \in {\bf N} \cup \{0\},
 \]
 where $M_1=[1+(2\pi/\varphi_1)]^{\alpha+1}$. 
 Therefore,
 \begin{align*}
  f(\varphi)
     & \le f(\varphi_0+2(n+1)\pi)
     + \sum_{k=N}^n \overline{a} M_1 (\varphi_0+2(k+1)\pi)^{-\alpha-1} \\
   & = f(\varphi_0+2(n+1)\pi)
     + \overline{a}M_1 \sum_{k=N}^n \int_k^{k+1}
       (\varphi_0+2(k+1)\pi)^{-\alpha-1} dt \\
  & \le f(\varphi_0+2(n+1)\pi)
     + \overline{a}M_1 \sum_{k=N}^n \int_k^{k+1}
       (\varphi_0+2\pi t)^{-\alpha-1} dt \\
   & = f(\varphi_0+2(n+1)\pi)
     + \overline{a}M_1 \int_N^{n+1} (\varphi_0+2\pi t)^{-\alpha-1} dt  \\
   & = f(\varphi_0+2(n+1)\pi)
     + \frac{\overline{a}M_1}{2\pi\alpha}
         \left[ (\varphi_0+2N\pi)^{-\alpha}
              - (\varphi_0+2(n+1)\pi)^{-\alpha} \right].
 \end{align*}
 Letting $n \to \infty$, we obtain
 \[
  f(\varphi)
  \le \frac{\overline{a}M_1}{2\pi\alpha} (\varphi_0+2N\pi)^{-\alpha}
  = \frac{\overline{a}M_1}{2\pi\alpha} \varphi^{-\alpha}.
 \]
\end{proof}

Hereafter, in this section, we assume that all assumptions of Theorem
\ref{criterion}.
Then, by Lemma \ref{f<}, there exists a positive constant $\overline{m}$
such that $f(\varphi) \le \overline{m} \varphi^{-\alpha}$ for
$\varphi \ge \varphi_1$.

Let $\varepsilon \in (0,1)$ be sufficiently small.
We use the following notation:
\[
 \varphi_2(\varepsilon)
  = \left( \frac{2 \overline{a}}{\varepsilon} \right)^{\frac{1}{\alpha+1}};
\]
\begin{equation*}
 \Gamma(\psi_1,\psi_2) 
  = \{ (f(\varphi)\cos\varphi,f(\varphi)\sin\varphi) : 
       \psi_1 \le \varphi < \psi_2 \};
\end{equation*}
\[
 T(\Gamma,\varepsilon) 
  = \Gamma(\varphi_1,\varphi_2(\varepsilon))_\varepsilon;
\]
\[
 N(\Gamma,\varepsilon) 
  = \Gamma(\varphi_2(\varepsilon),\infty)_\varepsilon,
\]
where $\Gamma_\varepsilon$ denotes the $\varepsilon$-neighborhood of 
$\Gamma$ defined by \eqref{Ge}.
Then, $\Gamma_\varepsilon = T(\Gamma,\varepsilon) \cup N(\Gamma,\varepsilon)$.

\begin{lem}\label{A<N}
 \[
 \{(r\cos\varphi,r\sin\varphi) :
  0 \le r \le f(\varphi), \
  \varphi \in [\varphi_2(\varepsilon),\varphi_2(\varepsilon)+2\pi) \}
  \subset N(\Gamma,\varepsilon).  
 \]
\end{lem}

\begin{proof}
 Let 
 \[
  (x_0,y_0) \in \{(r\cos\varphi,r\sin\varphi) :
  0 \le r \le f(\varphi), \
  \varphi \in [\varphi_2(\varepsilon),\varphi_2(\varepsilon)+2\pi) \}.
 \]
 Set $r_0=\sqrt{x_0^2+y_0^2}$.
 Then, there exists $\varphi_0 \ge \varphi_2(\varepsilon)$ such that
 $(x_0,y_0)=(r_0\cos\varphi_0,r_0\sin\varphi_0)$ and
 \[
  f(\varphi_0+2\pi) \le r_0 \le f(\varphi_0).
 \]
 We have
 \begin{align*}
  0 \le f(\varphi_0) - r_0
  \le f(\varphi_0) - f(\varphi_0+2\pi)
  \le \overline{a} \varphi_0^{-\alpha-1}
  \le \overline{a} (\varphi_2(\varepsilon))^{-\alpha-1}
  = \frac{\varepsilon}{2}.
 \end{align*}
 Therefore, 
 \[
  d((x_0,y_0),(f(\varphi_0)\cos\varphi_0,f(\varphi_0)\sin\varphi_0))
  = f(\varphi_0) - r_0 < \varepsilon,
 \]
 which means that $(x_0,y_0) \in N(\Gamma,\varepsilon)$.
\end{proof}

\begin{lem}\label{<|N|<}
 \[
  \pi \underline{m}^2
  \left[(2\overline{a})^{\frac{1}{\alpha+1}}+2\pi \right]^{-2\alpha}
  \varepsilon^{\frac{2\alpha}{\alpha+1}} \le |N(\Gamma,\varepsilon)| \le
  \pi \left[ \overline{m} (2\overline{a})^{-\frac{\alpha}{\alpha+1}} + 1
  \right]^2 \varepsilon^{\frac{2\alpha}{\alpha+1}}.
 \]
\end{lem}

\begin{proof}
 Set
 \[
  r_*(\varepsilon)
  = \min_{\psi \in [\varphi_2(\varepsilon),\varphi_2(\varepsilon)+2\pi]}
   f(\psi),
  \quad
  r^*(\varepsilon)
  = \max_{\psi \in [\varphi_2(\varepsilon),\varphi_2(\varepsilon)+2\pi]}
 f(\psi),
 \]
 and
 \[
  A=\{(r\cos\varphi,r\sin\varphi) :
  0 \le r \le f(\varphi), \
  \varphi \in [\varphi_2(\varepsilon),\varphi_2(\varepsilon)+2\pi) \}.
 \]
 Then, we easily find that
 \[
  \{ (r\cos\varphi,r\sin\varphi): 0 \le r \le r_*(\varepsilon), \
                                  \varphi \in {\bf R} \}
  \subset A.
 \]
 Therefore, Lemma \ref{A<N} implies that
 \begin{align*}
  |N(\Gamma,\varepsilon)|
  & \ge |A| \\
  & \ge \pi (r_*(\varepsilon))^2 \\
  & \ge \pi \left( \min_{\psi \in
  [\varphi_2(\varepsilon),\varphi_2(\varepsilon)+2\pi]} \underline{m}
  \psi^{-\alpha} \right)^2 \\
  & = \pi \underline{m}^2 (\varphi_2(\varepsilon)+2\pi)^{-2\alpha} \\
  & = \pi \underline{m}^2
  \left[(2\overline{a})^{\frac{1}{\alpha+1}}+2\pi
  \varepsilon^{\frac{1}{\alpha+1}} \right]^{-2\alpha}
  \varepsilon^{\frac{2\alpha}{\alpha+1}} \\
  & \ge \pi \underline{m}^2
  \left[(2\overline{a})^{\frac{1}{\alpha+1}}+2\pi \right]^{-2\alpha}
  \varepsilon^{\frac{2\alpha}{\alpha+1}},
 \end{align*}
 since $\varepsilon \in (0,1)$.

 Let $(x,y) \in N(\Gamma,\varepsilon)$.
 Then, there exists $(x_0,y_0) \in \Gamma(\varphi_2(\varepsilon),\infty)$
 and
 \[
  d((x,y),(x_0,y_0))<\varepsilon.
 \]
 Hence,
 \[
  d((x,y),(0,0))
  \le d((x,y),(x_0,y_0)) + d((x_0,y_0),(0,0))
  < \varepsilon + r^*(\varepsilon).
 \]
 It follows that
 \begin{align*}
  |N(\Gamma,\varepsilon)|
  & \le \pi (\varepsilon + r^*(\varepsilon))^2 \\
  & \le \pi \left( \varepsilon
   + \max_{\psi \in [\varphi_2(\varepsilon),\varphi_2(\varepsilon)+2\pi]}
  \overline{m} \psi^{-\alpha} \right)^2 \\
  & = \pi \left[ \varepsilon
  + \overline{m} (\varphi_2(\varepsilon))^{-\alpha} \right]^2 \\
  & = \pi \left[ \varepsilon^{\frac{1}{\alpha+1}}
  + \overline{m} (2\overline{a})^{-\frac{\alpha}{\alpha+1}} \right]^2
    \varepsilon^{\frac{2\alpha}{\alpha+1}} \\
  & \le \pi \left[ 1 + \overline{m} (2\overline{a})^{-\frac{\alpha}{\alpha+1}}
  \right]^2 \varepsilon^{\frac{2\alpha}{\alpha+1}}.
 \end{align*}
\end{proof}

\begin{lem}\label{|Ge|<}
 Let $x$, $y \in C[a,b]$ and let
 \[
  G = \{ (x(s),y(s)) : a \le s \le b \}.
 \]
 Assume that $(x(s),y(s))\ne(x(t),y(t))$ for $a \le s < t \le b$.
 Then,
 \[
  |G_\varepsilon|
   \le 4 \pi \varepsilon \,\mbox{\rm length}(G)
     + 4 \pi \varepsilon^2, \quad \varepsilon>0.
 \]
\end{lem}

\begin{proof}
 The proof is similar to the proof of Lemma 26 in \cite{PT13}.
 Let $\varepsilon>0$. 
 Set $s_1=a$ and
 \[
  s_{i+1} = \max \{ s \in [s_i,b] : d((x(t),y(t)),(x(s_i),y(s_i))) \le
  \varepsilon, \ t \in [s_i,s] \}
 \]
 for $i= 1,2,\cdots$.
 Then, there exists $n \ge 2$ such that $s_n=b$.
 Set $N=\max\{ i \in {\bf N} : s_i < b  \}$.
 We find that $N \ge 1$,
 \[
  a = s_1 < s_2 < \cdots < s_i < s_{i+1} < \cdots < s_N < s_{N+1}= b,
 \]
 and if $N \ge 2$, then
 \[
  d((x(s_i),y(s_i)),(x(s_{i+1}),y(s_{i+1}))) = \varepsilon, \quad
  i= 1,2,\cdots,N-1.
 \]
 We will prove that
 \begin{equation}
  G_\varepsilon \subset \bigcup_{i=1}^{N} B_{2\varepsilon}(x(s_i),y(s_i)),
  \label{Ge<B2e}
 \end{equation}
 where 
 \[ 
  B_{2\varepsilon}(x_0,y_0) = 
  \{ (x,y) \in {\bf R}^2 : d((x_0,y_0),(x,y)) \le 2 \varepsilon \}.
 \]
 Let $(x_1,y_1) \in G_\varepsilon$.
 Then, there exists $\sigma \in [a,b]$ such that
 \[
  d((x_1,y_1),(x(\sigma),y(\sigma))) \le \varepsilon.  
 \]
 Because of the definition of $s_i$, we find that
 $\sigma \in [s_k,s_{k+1}]$ for some $k \in \{1,2,\cdots,N\}$, which
 implies that
 \[
  d((x(\sigma),y(\sigma)),(x(s_k),y(s_k))) \le \varepsilon.
 \]
 Hence, it follows that
 \begin{multline*}
  d((x_1,y_1),(x(s_k),y(s_k))) \\
  \le d((x_1,y_1),(x(\sigma),y(\sigma)))
    + d((x(\sigma),y(\sigma)),(x(s_k),y(s_k)))
  \le 2 \varepsilon,
 \end{multline*}
 which means that $(x_1,y_1) \in B_{2\varepsilon}(x(s_k),y(s_k))$.
 Therefore, we obtain \eqref{Ge<B2e}.
 By \eqref{Ge<B2e}, we conclude that
 \begin{equation}
 |G_\varepsilon|
  \le \sum_{i=1}^{N} |B_{2\varepsilon}(x(s_i),y(s_i))|
  = 4 N \pi \varepsilon^2.
   \label{Ge(y)<4Npie2}
 \end{equation}
 When $N=1$, from \eqref{Ge(y)<4Npie2} it follows that
 \[
 |G_\varepsilon| 
 \le 4 \pi \varepsilon^2
 \le 4 \pi \varepsilon \,\mbox{length}(G) + 4\pi \varepsilon^2.
 \]
 Now, we assume that $N\ge 2$.
 We observe that
 \begin{align*}
  \mbox{length}(G)
  & \ge \sum_{i=1}^N d((x(s_i),y(s_i)),(x(s_{i+1}),y(s_{i+1}))) \\
  & \ge \sum_{i=1}^{N-1} d((x(s_i),y(s_i)),(x(s_{i+1}),y(s_{i+1}))) \\
  & = (N-1) \varepsilon,
 \end{align*}
 that is,
 \begin{equation}
  N \varepsilon \le \mbox{length}(G) + \varepsilon.
   \label{Ne<L+e}
 \end{equation}
 Combining \eqref{Ge(y)<4Npie2} with \eqref{Ne<L+e}, we obtain
 \[
 |G_\varepsilon|
 \le 4 \pi \varepsilon \,\mbox{length}(G) + 4\pi \varepsilon^2.
\]
\end{proof}

\begin{lem}\label{|T|<}
 \[
 |T(\Gamma,\varepsilon)|
 \le 4 \pi \left[ M (2 \overline{a})^{\frac{1-\alpha}{\alpha+1}} + 1 \right]
  \varepsilon^{\frac{2\alpha}{\alpha+1}}. 
 \]
\end{lem}

\begin{proof}
 From Lemma \ref{|Ge|<}, it follows that
 \begin{align*}
  |T(\Gamma,\varepsilon)| 
   & \le 4 \pi \varepsilon
   \,\mbox{length}(\Gamma(\varphi_1,\varphi_2(\varepsilon))) 
   + 4\pi \varepsilon^2 \\
   & \le 4 \pi \varepsilon M (\varphi_2(\varepsilon))^{1-\alpha} 
   + 4\pi \varepsilon^2 \\
  & = 4 \pi M (2 \overline{a})^{\frac{1-\alpha}{\alpha+1}} 
   \varepsilon^{\frac{2\alpha}{\alpha+1}}
  + 4\pi \varepsilon^2 \\
  & = 4 \pi \left[ M (2 \overline{a})^{\frac{1-\alpha}{\alpha+1}} 
  + \varepsilon^{\frac{2}{\alpha+1}} \right]
  \varepsilon^{\frac{2\alpha}{\alpha+1}} \\
  & \le 4 \pi \left[ M (2 \overline{a})^{\frac{1-\alpha}{\alpha+1}} + 1 \right]
  \varepsilon^{\frac{2\alpha}{\alpha+1}}.
 \end{align*} 
\end{proof}

Now, we are ready to prove Theorem \ref{criterion}.
 
\begin{proof}[Proof of Theorem \ref{criterion}]
 Since  
 \[
  |\Gamma_\varepsilon| \ge |N(\Gamma,\varepsilon)|
 \]
 and
 \[
  |\Gamma_\varepsilon| \le |T(\Gamma,\varepsilon)| + |N(\Gamma,\varepsilon)|,
 \]
 Lemmas \ref{<|N|<} and \ref{|T|<} imply that there exist positive
 constants $C_1$ and $C_2$ such that
 \[
  C_1 \varepsilon^{\frac{2\alpha}{\alpha+1}}
   \le |\Gamma_\varepsilon| \le 
  C_2 \varepsilon^{\frac{2\alpha}{\alpha+1}}
 \]
 for all sufficiently small $\varepsilon \in (0,1)$.
 Consequently, $\dim_{\rm B} \Gamma = 2/(1+\alpha)$.
\end{proof}

\begin{proof}[Proof of Corollary \ref{cor_criterion}]
 Let $\varphi \ge \varphi_1$ be fixed.
 Since $f'(\varphi) \le 0$ and $f'(\varphi)\not\equiv 0$ on
 $[\varphi,\varphi+2\pi)$, we have
 \[
  0 > \int_\varphi^{\varphi+2\pi} f'(\psi) d\psi
    = f(\varphi+2\pi) - f(\varphi).
 \]
 By the mean value theorem, there exists $c \in (\varphi,\varphi+2\pi)$
 \[
  \frac{f(\varphi+2\pi) - f(\varphi)}{2\pi} = f'(c),
 \]
 which implies that
 \[
  f(\varphi)-f(\varphi+2\pi) 
  = - 2\pi f'(c) 
  \le 2\pi K c^{-\alpha-1}
  \le 2\pi K \varphi^{-\alpha-1}.
 \]
 Then, by Lemma \ref{f<}, there exists a positive constant $\overline{m}$
 such that $f(\psi) \le \overline{m} \psi^{-\alpha}$ for $\psi \ge \varphi_1$.
 Therefore,
 \begin{align*}
  \mbox{\rm length}(\Gamma(\varphi_1,\varphi)) 
  & = \int_{\varphi_1}^\varphi \sqrt{(f(\psi))^2+(f'(\psi))^2} d\psi \\
  & \le \int_{\varphi_1}^\varphi
    \sqrt{(\overline{m}\psi^{-\alpha})^2+(K\psi^{-\alpha-1})^2} d\psi \\
  & = \int_{\varphi_1}^\varphi \psi^{-\alpha}
    \sqrt{\overline{m}^2+K^2\psi^{-2}} d\psi \\
  & \le \sqrt{\overline{m}^2+K^2\varphi_1^{-2}} 
   \int_{\varphi_1}^\varphi \psi^{-\alpha} d\psi \\
  & = \frac{\sqrt{\overline{m}^2+K^2\varphi_1^{-2}}}{1-\alpha}
   ( \varphi^{1-\alpha} - \varphi_1^{1-\alpha} ) \\
  & \le \frac{\sqrt{\overline{m}^2+K^2\varphi_1^{-2}}}{1-\alpha}
    \varphi^{1-\alpha}.
 \end{align*}
 Theorem \ref{criterion} implies that $\dim_{\rm B} \Gamma = 2/(1+\alpha)$.
\end{proof}

% Section 3

\section{Spiral with the box-counting dimension one}

In this section, we prove Theorem \ref{criterion_one} and assume that
all assumptions of Theorem \ref{criterion_one}.
Let $\varepsilon \in (0,\varphi_1^{-2})$ be sufficiently small.
We use the following notation:
\[
 T_1(\Gamma,\varepsilon) 
  = \Gamma(\varphi_1,\varepsilon^{-1/2})_\varepsilon;
\]
\[
 N_1(\Gamma,\varepsilon) 
  = \Gamma(\varepsilon^{-1/2},\infty)_\varepsilon.
\]
where $\Gamma(\psi_1,\psi_2) 
  = \{ (f(\varphi)\cos\varphi,f(\varphi)\sin\varphi) : 
       \psi_1 \le \varphi < \psi_2 \}$.
In the same way of the proof of Lemma \ref{<|N|<}, we have the following
result.

\begin{lem}\label{|N1|<}
 $|N_1(\Gamma,\varepsilon)| \le \pi (\overline{m} + 1)^2 \varepsilon$.
\end{lem}

\begin{lem}\label{|T1|<}
 $|T_1(\Gamma,\varepsilon)|
 \le -2 \pi M  \varepsilon \log \varepsilon + 4\pi \varepsilon^2$.
\end{lem}

\begin{proof}
 By Lemma \ref{|Ge|<}, we find that
 \begin{align*}
  |T_1(\Gamma,\varepsilon)| 
   & \le 4 \pi \varepsilon
   \,\mbox{length}(\Gamma(\varphi_1,\varepsilon^{-1/2})) 
   + 4\pi \varepsilon^2 \\
   & \le 4 \pi M \varepsilon
   \log \varepsilon^{-1/2} + 4\pi \varepsilon^2 \\
   & = -2 \pi M  \varepsilon \log \varepsilon + 4\pi \varepsilon^2.
 \end{align*} 
\end{proof}

The following inequality has been obtained in Tricot \cite[\S 9.1]{Tri}.

\begin{lem}\label{|Ge|>}
 Let $G$ be a curve in ${\bf R}^2$ and let $\mbox{\rm diam}(G)$ be the
 largest distance between each two points in $G$, that is
 \[
  \mbox{\rm diam}(G) = \sup_{z, w \in G} d(z,w).
 \]
 Assume that $\mbox{\rm diam}(G)<\infty$.
 Then,
 \[
  |G_\varepsilon| \ge 2 \varepsilon \,\mbox{\rm diam}(G) + \pi \varepsilon^2.
 \]
\end{lem}

Now, we give a proof of Theorem \ref{criterion_one}.

\begin{proof}[Proof of Theorem \ref{criterion_one}]
 Since the distance between two points
 \[
  (f(\varphi_1)\cos\varphi_1,f(\varphi_1)\sin \varphi_1)
 \]
 and
 \[
  (f(\varphi_1+\pi)\cos(\varphi_1+\pi),f(\varphi_1+\pi)\sin(\varphi_1+\pi))
 \]
 is equal to $f(\varphi_1)+f(\varphi_1+\pi)$,
 we have
 \[
  \mbox{\rm diam}(\Gamma) \ge f(\varphi_1)+f(\varphi_1+\pi).
 \]
 Hence, from Lemma \ref{|Ge|>}, it follows that
 \begin{align*}
  |\Gamma_\varepsilon|
  \ge 2 \varepsilon \,\mbox{\rm diam}(\Gamma) + \pi \varepsilon^2 
  \ge 2 (f(\varphi_1)+f(\varphi_1+\pi)) \varepsilon,
 \end{align*}
 which implies that
 \begin{align*}
  \liminf_{\varepsilon\to+0}
  \frac{\log|\Gamma_\varepsilon|}{\log \varepsilon}
  & \ge \liminf_{\varepsilon\to+0}
  \frac{\log (f(\varphi_1)+f(\varphi_1+\pi))\varepsilon}
  {\log \varepsilon} \\
  & = \liminf_{\varepsilon\to+0} \left(
  \frac{\log (f(\varphi_1)+f(\varphi_1+\pi))}{\log \varepsilon} + 1 \right) = 1.
 \end{align*}
 By Lemmas \ref{|N1|<} and \ref{|T1|<}, we conclude that
 \begin{align*}
  |\Gamma_\varepsilon|
   & \le |T_1(\Gamma,\varepsilon)| + |N_1(\Gamma,\varepsilon)| \\
   & \le -2 \pi M  \varepsilon \log \varepsilon + 4\pi \varepsilon^2
      + \pi (\overline{m} + 1)^2 \varepsilon \\
   & = [-2 \pi M \log \varepsilon + 4\pi \varepsilon
     + \pi (\overline{m} + 1)^2] \varepsilon \\
   & \le [-2 \pi M \log \varepsilon + 4\pi
   + \pi (\overline{m} + 1)^2] \varepsilon,
 \end{align*}
 since $\varepsilon \in (0,1)$.
 Therefore,
 \[
  |\Gamma_\varepsilon| \le (-c_1 \log \varepsilon + c_2) \varepsilon
 \]
 for some $c_1>0$ and $c_2>0$, which implies that
 \begin{align*}
  \limsup_{\varepsilon\to+0}
  \frac{\log|\Gamma_\varepsilon|}{\log \varepsilon}
  & \le \limsup_{\varepsilon\to+0}
     \frac{\log(-c_1 \log \varepsilon + c_2) \varepsilon}
          {\log\varepsilon} \\
  & = \limsup_{\varepsilon\to+0}
   \left(
    \frac{\log(-c_1 \log \varepsilon + c_2)}{\log\varepsilon} + 1
  \right)
  = 1.
 \end{align*}
 Consequently, $\dim_{\rm B} \Gamma = 1$.
\end{proof}

\begin{proof}[Proof of Corollary \ref{criterion_one_cor}]
 Let $\varphi \ge \varphi_1$ be fixed.
 By the same argument as in the proof of Corollary \ref{cor_criterion},
 we find that $0 < f(\varphi) - f(\varphi+2\pi)$.
 We observe that
 \begin{align*}
  \mbox{\rm length}(\Gamma(\varphi_1,\varphi)) 
  & = \int_{\varphi_1}^\varphi \sqrt{(f(\psi))^2+(f'(\psi))^2} d\psi \\
  & \le \int_{\varphi_1}^\varphi
    \sqrt{(\overline{m}\psi^{-1})^2+(K\psi^{-1})^2} d\psi \\
  & = \sqrt{\overline{m}^2+K^2} \int_{\varphi_1}^\varphi \psi^{-1} d\psi \\
  & = \sqrt{\overline{m}^2+K^2} ( \log \varphi - \log \varphi_1 ) \\
  & \le \sqrt{\overline{m}^2+K^2} \log \varphi,
 \end{align*}
 since $\varphi_1>1$.
 Applying Theorem \ref{criterion_one}, we conclude that $\dim_{\rm B}\Gamma=1$.
\end{proof}

% Section 4

\section{Box-counting dimension of solution curves}

In this section, we give proofs of Theorems \ref{F} and \ref{F1}.

For each solution $(x(t),y(t))$ of \eqref{S}, we use the following notation:
\[
 r(t) = \sqrt{|x(t)|^2 + |y(t)|^2}.
\]
The following Lemmas \ref{V}, \ref{r'theta'} and \ref{h->0} have been
obtained in
\cite[Lemmas 2.2, 3.1 and 4.2]{OT2017}.

\begin{lem}\label{V}
 Let $(x(t),y(t))$ be a nontrivial solution of \eqref{S}.
 Assume that \eqref{HW} is satisfied.
 Then, there exist a constant $C>0$ and a function $\delta \in C[t_0,\infty)$
 such that $\lim_{t\to\infty}\delta(t)=0$ and
 \begin{equation*}
  [r(t)]^2 = e^{-H(t)} [C + \delta(t)], \quad t \ge t_0.
 \end{equation*}
\end{lem}

\begin{lem}\label{r'theta'}
 Let $(x(t),y(t))$ be a nontrivial solution of \eqref{S}.
 If $x(t)=r(t)\cos \theta(t)$ and $y(t)=r(t)\sin \theta(t)$, then
 \begin{equation*}
  \left\{
   \begin{array}{l}
    r'(t) = -h(t) r(t) \sin^2 \theta(t), \\[1ex]
     \theta'(t) = -1 - \displaystyle\frac{1}{2} h(t) \sin 2 \theta(t). 
   \end{array}
	\right.
 \end{equation*}
\end{lem}

\begin{lem}\label{h->0}
 If \eqref{HW} is satisfied, then $\lim_{t\to\infty} h(t)=0$. 
\end{lem}

\begin{proof}[Proof of Theorem \ref{F}]
 Let $(x(t),y(t))$ be a nontrivial solution of \eqref{S}.
 We note that \eqref{H(oo)=oo} holds, by \eqref{H(t)=2alog}. 
 From Theorem A, it follows that
 $\lim_{t\to\infty}x(t)=\lim_{t\to\infty}y(t)=0$,
 $(x(t),y(t))$ is a spiral, rotating in a
 clockwise direction on $[t_1,\infty)$ for some $t_1 \ge t_0$
 and $\Gamma_{(x,y;t_0)}$ is simple.
 By l'Hopital's rule and Lemmas \ref{r'theta'} and \ref{h->0}, we have
 \begin{equation}\label{theta'->-1}
  \lim_{t\to\infty} \frac{\theta(t)}{t} = \lim_{t\to\infty} \theta'(t)= -1.
 \end{equation}
 Since
 \[
  t^\alpha r(t) = t^\alpha e^{-H(t)/2} \sqrt{e^{H(t)}[r(t)]^2}
  = e^{-\frac{1}{2}(H(t)-2\alpha\log t)} \sqrt{e^{H(t)}[r(t)]^2},
 \]
 Lemma \ref{V} and \eqref{H(t)=2alog} imply that
 \begin{equation}
  0 < \liminf_{t\to\infty} t^\alpha r(t) \le
      \limsup_{t\to\infty} t^\alpha r(t) < \infty.
   \label{0<tar<oo}    
 \end{equation} 
 By \eqref{theta'->-1}, \eqref{0<tar<oo} and \eqref{th(t)<}, there exist
 $t_2 \ge \max\{t_1,1\}$, $C_1>0$, $C_2>0$ and $C_3>0$ such that, for
 $t\ge t_2$,
 \begin{gather}
  - \frac{3}{2} t \le \theta(t) \le -\frac{1}{2} t,
  \label{<theta<} \\
  - \frac{3}{2} \le \theta'(t) \le -\frac{1}{2},
  \label{<theta'<} \\
  C_1 \le t^\alpha r(t) \le C_2, 
  \label{<tar<} \\
  th(t) \le C_3.
  \label{th(t)<C}
 \end{gather}
 In view of \eqref{<theta<}, we note that $\lim_{t\to\infty} \theta(t)=-\infty$.
 Set $\eta(t)=-\theta(t)$.
 Then $\eta$ is positive and strictly increasing on $[t_2,\infty)$.
 Hence, $\eta$ has the inverse function $\eta^{-1}$.
 Set $\varphi_2=\eta(t_2)>0$ and $f(\varphi)=r(\eta^{-1}(\varphi))$ on
 $[\varphi_2,\infty)$.
 Since $\lim_{t\to\infty}x(t)=\lim_{t\to\infty}y(t)=0$,
 we have $\lim_{t\to\infty}r(t)=0$, and hence,
 $\lim_{\varphi\to\infty}f(\varphi)=0$.
 From \eqref{<theta<} and \eqref{<tar<}, it follows that
 \begin{align*}
  \varphi^\alpha f(\varphi)
  = \varphi^\alpha r(\eta^{-1}(\varphi))
  = (\eta(t))^\alpha r(t) 
  & = \left( \frac{-\theta(t)}{t} \right)^\alpha t^\alpha r(t) \\
  & \ge \frac{C_1}{2^\alpha}, \quad \varphi \ge \varphi_2,  
 \end{align*}
 where $t=\eta^{-1}(\varphi)$.
 By \eqref{<theta'<} and Lemma \ref{r'theta'}, we find that
 \begin{align}
  f'(\varphi) 
   & = r'(\eta^{-1}(\varphi)) \frac{1}{\eta'(\eta^{-1}(\varphi))}
       \label{f'=} \\
   & = -\frac{r'(t)}{\theta'(t)}  \nonumber \\
   & = \frac{h(t) r(t) \sin^2 \theta(t)}{\theta'(t)} \le 0, 
     \quad \varphi \ge \varphi_2, \nonumber
 \end{align} 
 where $t=\eta^{-1}(\varphi)$.
 We conclude that $f'(\varphi)\not\equiv 0$ on $[\varphi,\varphi+2\pi)$
 for each fixed $\varphi \ge \varphi_2$.
 Indeed, if $f'(\varphi)\equiv 0$ on $[\varphi,\varphi+2\pi)$
 for some $\varphi \ge \varphi_2$, then, by \eqref{f'=},
 $\sin^2 \theta(t) \equiv 0$ on
 $I:=[\eta^{-1}(\varphi),\eta^{-1}(\varphi+2\pi))$, that is,
 that $\theta'(t)\equiv 0$ on $I$, which contradicts \eqref{<theta'<}.
 Combining \eqref{<theta<}, \eqref{<tar<}, \eqref{th(t)<C} with
 \eqref{f'=}, we find that
 \begin{align*}
  - \varphi^{\alpha+1} f'(\varphi)
   & = (\eta(t))^{\alpha+1} \frac{h(t) r(t) \sin^2 \theta(t)}{-\theta'(t)} \\
   & = \left( \frac{-\theta(t)}{t} \right)^{\alpha+1}
     \frac{t^{\alpha+1}h(t) r(t) \sin^2 \theta(t)}{-\theta'(t)} \\
   & \le \left( \frac{3}{2} \right)^{\alpha+1} 2 C_2C_3,
     \quad \varphi \ge \varphi_2,
 \end{align*}
 where $t=\eta^{-1}(\varphi)$.
 Set
 \[
  \Gamma = \{ (f(\varphi)\cos \varphi, f(\varphi)\sin \varphi)
           : \varphi \ge \varphi_2 \}.
 \]
 Corollary \ref{cor_criterion} implies that $\dim_{\rm B} \Gamma =2/(1+\alpha)$.
 Since
 \begin{align*}
  \Gamma_{(x,-y;t_2)}
    & = \{ (x(t),-y(t)) : t \ge t_2 \} \\
    & = \{ (r(t)\cos \theta(t),-r(t)\sin \theta(t)) : t \ge t_2 \} \\
    & = \{ (r(\eta^{-1}(\varphi))\cos \theta(\eta^{-1}(\varphi)),
           -r(\eta^{-1}(\varphi))\sin \theta(\eta^{-1}(\varphi)))
            : \varphi \ge \varphi_2 \} \\
    & = \{ (f(\varphi)\cos (-\varphi),
            -f(\varphi)\sin (-\varphi))
            : \varphi \ge \varphi_2 \} \\
    & = \{ (f(\varphi)\cos \varphi, f(\varphi)\sin \varphi)
            : \varphi \ge \varphi_2 \} \\
    & = \Gamma,  
 \end{align*}
 we have $\dim_{\rm B} \Gamma_{(x,-y;t_2)}=2/(1+\alpha)$.
 Since, $\Gamma_{(x,y;t_2)}$ and $\Gamma_{(x,-y;t_2)}$ are symmetric, we
 conclude that
 \[
     \dim_{\rm B} \Gamma_{(x,y;t_2)}
   = \dim_{\rm B} \Gamma_{(x,-y;t_2)}
   = \dim_{\rm B} \Gamma
   = \frac{2}{1+\alpha}.
 \]
\end{proof}

\begin{proof}[Proof of Theorem \ref{F1}]
 Let $(x(t),y(t))$ be a nontrivial solution of \eqref{S}.
 Using \eqref{H(t)=2log}, we have \eqref{H(oo)=oo}.
 Hence, from Theorem A, it follows that
 $\lim_{t\to\infty}x(t)=\lim_{t\to\infty}y(t)=0$,
 $(x(t),y(t))$ is a spiral, rotating in a clockwise direction on
 $[t_1,\infty)$ for some $t_1 \ge t_0$ and $\Gamma_{(x,y;t_0)}$ is simple.
 By the same argument as in the proof of Theorem \ref{F} and noting
 Lemma \ref{h->0}, there exist
 $t_2 \ge \max\{t_1,1\}$, $C_1>0$, $C_2>0$ and $C_3>0$ such that
 \eqref{<theta<}, \eqref{<theta'<} and the following \eqref{<tr<} and
 \eqref{h(t)<C} hold for $t\ge t_2$:
 \begin{gather}
  C_1 \le t r(t) \le C_2, \label{<tr<} \\
  h(t) \le C_3. \label{h(t)<C}
 \end{gather}
 Set $\eta(t)=-\theta(t)$.
 Then, $\eta$ has the inverse function $\eta^{-1}$.
 Set $\varphi_2=\eta(t_2)>0$ and
 $f(\varphi)=r(\eta^{-1}(\varphi))$ on $[\varphi_2,\infty)$.
 Then, $\lim_{\varphi\to\infty}f(\varphi)=0$.
 We observe that 
 \begin{equation*}
  \varphi f(\varphi)
  = \varphi r(\eta^{-1}(\varphi))
  = \left( \frac{-\theta(t)}{t} \right) t r(t) 
  \le \frac{3C_2}{2}, \quad \varphi \ge \varphi_2,  
 \end{equation*}
 where $t=\eta^{-1}(\varphi)$.
 In the same way as in the poof of Theorem \ref{F}, using
 \eqref{<theta<}, \eqref{<theta'<}, \eqref{f'=}, \eqref{<tr<} and
 \eqref{h(t)<C}, we conclude that $f'(\varphi) \le 0$ for
 $\varphi \ge \varphi_2$, $f'(\varphi)\not\equiv 0$ on
 $[\varphi,\varphi+2\pi)$ for each fixed $\varphi \ge \varphi_2$, and that 
 \begin{equation*}
  - \varphi f'(\varphi)
    = \left( \frac{-\theta(t)}{t} \right)
     \frac{h(t) tr(t) \sin^2 \theta(t)}{-\theta'(t)} \\
    \le 3C_2C_3,
     \quad \varphi \ge \varphi_2,
 \end{equation*}
 where $t=\eta^{-1}(\varphi)$.
 Corollary \ref{criterion_one_cor} implies that
 $\dim_{\rm B} \Gamma =1$.
 Consequently, $\dim_{\rm B} \Gamma_{(x,y;t_2)}=1$.
\end{proof}

\end{document}